\documentclass{amsart}
\usepackage{amsfonts,hyperref}

\newtheorem{theorem}{Theorem}
\theoremstyle{plain}

\newtheorem{corollary}{Corollary}

\newtheorem{remark}{Remark}

\numberwithin{equation}{section}
\theoremstyle{theorem}

\begin{document}
\title[ Bernoulli Polynomials with a $q$ Parameter in Terms of $r$-Whitney Numbers]{
An Explicit Formula for Bernoulli Polynomials with a $q$ Parameter in Terms of $r$-Whitney Numbers}
\author{F. A. Shiha}
\address{Department of Mathematics Faculty of Science \\
Mansoura University Mansoura, 35516, EGYPT.}
\email{fshiha@yahoo.com, fshiha@mans.edu.eg}

\begin{abstract}
We define the Bernoulli polynomials with a $q$ parameter in terms of $r$-Whitney numbers of the second kind. Some algebraic properties and combinatorial identities of these polynomials are given. Also, we obtain several relations between the Cauchy and Bernoulli polynomials with a $q$ parameter in terms of $r$-Whitney numbers of both kinds. 
\end{abstract}

\maketitle

\bigskip AMS (2010) Subject Classification: 05A15, 05A19, 11B68, 11B73.


\section{Introduction}
\bigskip The Bernoulli polynomials $B_n(z)$, are defined by the generating function \cite{comtet74}
\begin{equation}
\sum_{n=0}^{\infty} B_n(z)\,\frac{t^n}{n!}=\frac{t\, e^{zt}}{e^t-1}.
\end{equation}
When $z=0$, $B_n=B_n(0)$ are called the Bernoulli numbers. Graham et al. \cite[P. 560]{grah94} represented $B_n$ also as
\begin{equation}\label{E:bs2}
B_n=\sum_{k=0}^n(-1)^k\, \frac{k!}{k+1}\, S(n,k),
\end{equation}
where $S(n,k)$ are the Stirling numbers of the second kind \cite{comtet74}.

For all integers $n, r \geq 0$, $B_n(r)$ can be written explicitly  in terms of $r$-Stirling numbers
of the second kind $S_r(n,k)$  \cite{mezo2016}
\begin{equation}\label{E:mb}
B_n(r)=\sum_{k=0}^n\frac{(-1)^k\,k!}{k+1}\:S_r(n+r,k+r).
\end{equation}
 Cenkci et al. \cite{tkom2015} introduced the poly-Bernoulli polynomials with a $q$ parameter $ B_{n,q}^{(k)}(z)$ as a generalization of the poly-Bernoulli polynomials $ B_n^{(k)}(z)$. Let $q$ be a real number
with $q \neq 0$,  \cite{tkom2015} defined the poly-Bernoulli polynomials with a q parameter by
\begin{equation}\label{E:pbq}
\sum_{n=0}^{\infty} B_{n,q}^{(k)}(z)\,\frac{t^n}{n!}=\frac{q\,e^{-zt}}{1-e^{-qt}}\,\sum_{n=1}^{\infty}\left(\frac{(1-e^{-qt})}{q}\right)^n\,\frac{1}{n^k}.
\end{equation}
If $z = 0$, then  $B_{n,q}^{(k)}(0)=B_{n,q}^{(k)}$ are the poly-Bernoulli numbers with a $q$ parameter.
If $q=1$, then $ B_{n,1}^{(k)}(z)=B_n^{(k)}(z)$ are the poly-Bernoulli polynomials \cite{cop2010}.
 If $q=1$, $z=0$, then $ B_{n,1}^{(k)}(0)=B_n^{(k)}$ are the poly-Bernoulli numbers, defined in \cite{kan97}.

Setting $k=1$ in \eqref{E:pbq}, and let $ B_n^q(z)=B_{n,q}^{(1)}(z)$ denote the Bernoulli polynomials with $q$ parameter, we obtain
\begin{equation}
\sum_{n=0}^{\infty} B_{n}^{q}(z)\,\frac{t^n}{n!}=\frac{q\,e^{-zt}}{1-e^{-qt}}\,\sum_{n=1}^{\infty}\left(\frac{(1-e^{-qt})}{q}\right)^n\,\frac{1}{n}.
\end{equation}
If $z=0$, then $B_{n}^{q}(0)=B_{n}^{q}$ are the Bernoulli numbers with $q$ parameter. Note that $e^{zt}$ is replaced by $e^{-zt}$ (see \cite{tkom2015, bay2011, cop2010}).

In this paper, we are interested in Bernoulli polynomials with a $q$ parameter. New representation for these polynomials in terms of $r$-Whitney numbers of the second kind and many properties will be given.

For non-negative integers $n$ and $k$ with $0\leq k\leq n$, let $w(n,k)=w_{q,r}(n,k)$ and $W(n,k)=W_{q,r}(n,k)$ denote the $r$-Whitney numbers of the first and second kind, respectively, defined in \cite{mezo2010} by
\begin{equation}\label{E:wi1}
q^n(x)_n=\sum_{k=0}^n\:w(n,k)\:(qx+r)^k,
\end{equation}
\begin{equation}\label{E:wi2}
(qx+r)^n=\sum_{k=0}^{n}q^k\:W(n,k)\:(x)_k.
\end{equation}
The exponential generating function of $W(n,k)$ is given by \cite{mezo2010}
\begin{equation}\label{E:ewi2}
\sum_{n\geq k}W(n,k)\frac {t^n}{n!}=\frac{ e^{rt}}{ k!}\left(\frac{e^{qt}-1}{q}\right)^k.
\end{equation}
The parameters $r \geq 0$ and $q > 0$ are usually taken to be integers, but both may also be considered as indeterminates \cite{shat2017}.

Komatsu \cite{kom2013} defined the Cauchy polynomials with a $q$ parameter of the first and second kind, denoted by $c_n^q(z)$, $\hat{c}_n^q(z)$, respectively
\begin{equation}\label{E:komc1}
c_n^q(z)=\int_{0}^{1}(x-z|q)_n \, dx \qquad q\neq 0,
\end{equation}
\begin{equation}\label{E:komc2}
\hat{c}_n^q(z)=\int_{0}^{1}(-x+z|q)_n \, dx \qquad q\neq 0.
\end{equation}
Recently, Shiha \cite{shiha18} gave explicit formulae for  $c_n^q(r)$ and $\hat{c}_n^q(r)$ in terms of $r$-Whitney numbers of the first kind
\begin{equation}\label{D:cau1}
c_n^q(r)=\sum_{k=0}^{n}w(n,k)\:\frac{1}{k+1} \qquad q\neq 0,
\end{equation}
\begin{equation}\label{D:cau2}
\hat{c}_n^q(-r)=\sum_{k=0}^{n}(-1)^k\,w(n,k)\:\frac{1}{k+1} \qquad q\neq 0.
\end{equation}
The relations between $c_n^{q}(r)$, $\hat{c}_n^{q}(r)$ and $W(n,k)$ are (see \cite{shiha18})

\begin{equation}\label{E:Wc}
\sum_{k=0}^n W(n,k)\:c_k^{q}(r)=\frac{1}{n+1}
\end{equation}
\begin{equation}\label{E:Wc2}
\sum_{k=0}^n W(n,k)\:\hat{c}_k^{q}(-r)=\frac{(-1)^n}{n+1}
\end{equation}

\section{\protect\bigskip The basic results}

\begin{theorem}
 We define the Bernoulli polynomials with a $q$ parameter in terms of $W(n,k)$ by
 \begin{equation}\label{E:B}
 B_n^q(r)=\sum_{k=0}^{n}(-1)^k\:\frac{k!}{k+1}\:W(n,k).
 \end{equation}
 \end{theorem}
 \begin{proof}
 Using \eqref{E:ewi2}, we get
 \begin{equation*}
 \begin{split}
 \sum_{n=0}^{\infty}\,\sum_{k=0}^n(-1)^k\:\frac{k!}{k+1}W(n,k)\:\frac{t^n}{n!}&=\sum_{k=0}^{\infty}(-1)^k\:\frac{k!}{k+1}\sum_{n=k}^{\infty}W(n,k)\:\frac{t^n}{n!}\\
 &=\sum_{k=0}^{\infty}(-1)^k\:\frac{k!}{k+1}\:\frac{e^{rt}}{k!}\left(\frac{e^{qt}-1}{q}\right)^k\\
 &=\frac{q\;e^{rt}}{1-e^{qt}}\sum_{k=0}^{\infty}\left(\frac{1-e^{qt}}{q}\right)^{k+1}\:\frac{1}{k+1}\\
 &=\frac{q\:e^{rt}}{1-e^{qt}}\sum_{k=1}^{\infty}\left(\frac{1-e^{qt}}{q}\right)^k\:\frac{1}{k}\\
 &=\sum_{n=0}^{\infty}B_n^q(r)\:\frac{t^n}{n!}.
 \end{split}
 \end{equation*}
 Comparing the coefficients of both sides, we get \eqref{E:B} (notice that $e^{rt}$ is replaced by $e^{-rt}$ in \cite{tkom2015}).
 \end{proof}

 On the other hand,
 \begin{equation}\label{E:ebq}
 \begin{split}
 \sum_{n=0}^{\infty}B_n^q(r)\:\frac{t^n}{n!}&=\frac{q\:e^{rt}}{1-e^{qt}}\sum_{k=1}^{\infty}\left(\frac{1-e^{qt}}{q}\right)^k\:\frac{1}{k}\\
 &=\frac{q\:e^{rt}}{1-e^{qt}}\left(-\ln(1-(\frac{1-e^{qt}}{q}))\right)\\
 &=\frac{q\:e^{rt}}{e^{qt}-1}\left(\ln\frac{q-1+e^{qt}}{q}\right).
 \end{split}
 \end{equation}

 The first few polynomials are

 $B_0^q(r)=1$,

 $B_1^q(r)=r-\frac{1}{2}$,

 $B_2^q(r)=r^2-r-\frac{1}{2}q+\frac{2}{3}$,

 $B_3^q(r)=r^3-\frac{3}{2}r^2+(2-\frac{3}{2}q)r-\frac{1}{2}q^2+2q-\frac{3}{2}$,

 $B_4^q(r)=r^4-2r^3+(4-3q)r^2-2(q^2-4q+3)r-\frac{1}{2}q^3+\frac{14}{3}q^2-9q+\frac{24}{5}$.

 \begin{remark}
 If $r=0$, $B_n^q(0)=B_n^q$ are the Bernoulli numbers with  $q$ parameter, using \eqref{E:ebq}, we get

\begin{equation}
\sum_{n=0}^{\infty}B_n^q\:\frac{t^n}{n!}=\frac{q}{e^{qt}-1}\left(\ln\frac{q-1+e^{qt}}{q}\right),
\end{equation}
$ W_{q,0}(n,k)=q^{n-k}\,S(n,k)$, then $B_n^q$ can be expressed in terms of $S(n,k)$ as
\begin{equation}
 B_n^q=\sum_{k=0}^n (-1)^k\,\frac{k!}{k+1}\;q^{n-k}\,S(n,k),
 \end{equation}
If $q=1$, $B_n^1(r)=B_n(r)$ and $W_{1,r}(n,k)$ are reduced to $S_r(n+r,k+r)$ then, we get the explicit formula \eqref{E:mb}, and
\begin{equation}
\sum_{n=0}^{\infty} B_n(r)\,\frac{t^n}{n!}=\frac{t\, e^{rt}}{e^t-1}.
\end{equation}
If $q=1$ and $r=0$, then $B_n^1(0)=B_n$ and $W_{1,0}(n,k)$ are reduced to $S(n,k)$ then, we get the explicit formula \eqref{E:bs2}.
\end{remark}

It is known that the $r$-Whitney numbers satisfy the following orthogonality relation \cite{mezo2010}:
\begin{equation}
\sum_{l=k}^n\, w(n,l)\,W(l,k)=\sum_{l=k}^n\,W(n,l)\,w(l,k)=\delta_{kn},
\end{equation}
where $\delta_{kn}$ is the Kronecker delta.

\begin{corollary}
The relation between $w(n,k)$ and $B_n^q(r)$ is given by
\begin{equation}\label{E:wB}
\sum_{j=0}^n \, (-1)^n\, w(n,j)\,B_j^q(r)=\frac{n!}{n+1}.
\end{equation}
\end{corollary}

\begin{proof}
\begin{equation*}
\begin{split}
\sum_{j=0}^n\, (-1)^n\, w(n,j)\,B_j^q(r)&=\sum_{j=0}^n \,\sum_{k=0}^j\, (-1)^{n+k}\frac{k!}{k+1}\,w(n,j)\,W(j,k)\\
&=\sum_{k=0}^n (-1)^{n+k}\frac{k!}{k+1}\sum_{j=k}^n w(n,j)W(j,k)\\
&=\sum_{k=0}^n (-1)^{n+k}\frac{k!}{k+1}\delta_{kn}=\frac{n!}{n+1}.
\end{split}
\end{equation*}
\end{proof}

 The numbers $ W(n,k)$ are determined by (see \cite{mezo2010}).
 \begin{equation}
 W(n,k)=\frac{1}{q^k\:k!}\sum_{j=0}^{k}(-1)^{k-j}\binom{k}{j}(r+jq)^n,
 \end{equation}
 then, we obtain

 \begin{theorem}
 For $n\geq 0$ , we have
 \begin{equation}
 B_n^q(r)=\sum_{k=0}^n\frac{1}{q^k\:(k+1)}\sum_{j=0}^{k}(-1)^{j}\binom{k}{j}(r+jq)^n
 \end{equation}
 \end{theorem}

 \begin{theorem}
 For $n\geq 0$, we have
 \begin{equation}\label{E:cb}
 c_n^q(r)=\sum_{k=0}^n \sum_{j=0}^k \frac{(-1)^k}{k!}\:w(n,k)\,w(k,j)\:B_j^q(r),
 \end{equation}
  \begin{equation}\label{E:cb2}
 \hat{c}_n^q(-r)=\sum_{k=0}^n \sum_{j=0}^k \frac{1}{k!}\:w(n,k)\,w(k,j)\:B_j^q(r),
 \end{equation}
 \begin{equation}\label{E:bc}
 B_n^q(r)=\sum_{k=0}^{n}\sum_{j=0}^k\:(-1)^k\:k! \:W(n,k)\,W(k,j)\:c_j^q(r).
 \end{equation}
 \begin{equation}\label{E:bc2}
 B_n^q(r)=\sum_{k=0}^{n}\sum_{j=0}^k\:k! \:W(n,k)\,W(k,j)\:\hat{c}_j^q(-r).
 \end{equation}
 \end{theorem}

\begin{proof}
We shall prove the first and the third identities. Other identities are proven similarly. By \eqref{E:wB}, and using \eqref{D:cau1}, we have
 \begin{equation*}
 \begin{split}
 \sum_{k=0}^n\, \sum_{j=0}^k \frac{(-1)^k}{k!}\:w(n,k)\,w(k,j)\:B_j^q(r)&=\sum_{k=0}^n \frac{1}{k!}\, w(n,k)\sum_{j=0}^k (-1)^k\,w(k,j)\, B_j^q(r)\\
&=\sum_{k=0}^n \frac{1}{k+1}\,w(n,k)=c_n^q(r).
 \end{split}
 \end{equation*}
By \eqref{E:Wc}, and using \eqref{E:B}, we obtain
\begin{equation*}
\begin{split}
\sum_{k=0}^{n}\sum_{j=0}^k\:(-1)^k\:k! \,W(n,k)\,W(k,j)\,c_j^q(r)&=\sum_{k=0}^{n}(-1)^k\,k!\:W(n,k)\sum_{j=0}^k\,W(k,j)\,c_j^q(r)\\
&=\sum_{k=0}^{n}(-1)^k\,k!\:W(n,k)\,\frac{1}{k+1}=B_n^q(r).
\end{split}
\end{equation*}
\end{proof}

It is known that (see \cite{cheon12})
\begin{equation}
W_{q,r+s}(n,k)=\sum_{j=k}^n \binom{n}{j}\,r^{n-j}\,W_{q,s}(j,k).
\end{equation}
Therefore, we get the following properties, which are similar to those for classical Bernoulli polynomials and numbers.

\begin{theorem}
For $n \geq 0$, we have
\begin{equation}\label{E:rrb}
B_n^q(r+s)=\sum_{j=0}^n\binom{n}{j}\,r^{n-j}B_j^q(s).
\end{equation}
\begin{equation}\label{E:rrbn}
 B_n^q(r)=\sum_{j=0}^n\binom{n}{j}\,r^{n-j}B_j^q.
\end{equation}
\end{theorem}

\begin{proof}
 \begin{equation*}
 \begin{split}
 B_n^q(r+s)&=\sum_{k=0}^n(-1)^k\:\frac{k!}{k+1}\, W_{q,r+s}(n,k)\\
 &=\sum_{k=0}^n \,\sum_{j=k}^n (-1)^k\:\frac{k!}{k+1}\binom{n}{j}\,r^{n-j}\,W_{q,s}(j,k)\\
 &=\sum_{j=0}^n\binom{n}{j}\,r^{n-j}\:\sum_{k=0}^{j}(-1)^k \,\frac{k!}{k+1}\,W_{q,s}(j,k)=\sum_{j=0}^n\binom{n}{j}\,r^{n-j}\,B_j^q(s).
 \end{split}
 \end{equation*}
 Setting $s=0$ in \eqref{E:rrb}, we get the second relation \eqref{E:rrbn}
 \end{proof}
 \begin{remark}
 Setting $q=1$ in \eqref{E:rrb} and \eqref{E:rrbn}, we obtain the following properties of $B_n(r)$ \cite{tem96}
 \begin{equation}
B_n(r+s)=\sum_{j=0}^n\binom{n}{j}\,r^{n-j}B_j(s),\; \text{and }\;B_n(r)=\sum_{j=0}^n\binom{n}{j}\:r^{n-j}\:B_j.
\end{equation}
 \end{remark}

\end{document}